\documentclass{amsart}
\newtheorem{theorem}{Theorem}[section]

\newtheorem{es}[theorem]{Example}
\newtheorem{remark}[theorem]{Remark}
\def\erre{{\rm I\!R}}
\def\RR{{\rm I\!R}}
\def\R{{\rm I\!R}}

\def\R{{\rm I\!R}}
\def\R{{\rm I\!R}}

\def\phi{\varphi}

\def\osc{\mathop{\rm osc}}
\title[Higher nonlocal problems]{Higher nonlocal problems with\\ bounded potential}
\author{Giovanni Molica Bisci}
\address[G. Molica Bisci]{Dipartimento P.A.U., Universit\`a  degli
Studi Mediterranea di Reggio Calabria, Salita Melissari - Feo di
Vito, 89124 Reggio Calabria, Italy} \email{gmolica@unirc.it}
\author{Du\v{s}an Repov\v{s}}
\address[D. Repov\v{s}]{Faculty of Education, and Faculty of Mathematics and Physics\\ University of Ljubljana, POB 2964, Ljubljana, Slovenia 1001}
\email{dusan.repovs@guest.arnes.si}
\thanks{{\it 2010 Mathematics Subject Classification.} Primary:
35J62, 35J92, 35J20;
Secondary: 35J15, 47J30.}
\keywords{Fractional equations; Multiple solutions; Critical points results.}
\thanks{Typeset by \LaTeX}
\begin{document}
\begin{abstract}
 The aim of this paper is to study a class of nonlocal fractional Laplacian equations depending on two real parameters. More precisely, by using an appropriate analytical context on fractional Sobolev spaces due to Servadei and Valdinoci, we establish the existence of three weak solutions for nonlocal fractional problems exploiting an abstract critical point result for smooth functionals. We emphasize that the dependence of the underlying equation from one of the real parameter is not necessarily of affine type.
\end{abstract}
\maketitle
%%%%%%%%%%%%%%%%%%%%%%%%%%%%%%%%%%%%%%%%%%%%%%%%%%%%%%%%%%%%%%%%%%%%%%%%%%%%%%%%%%%%%%%%%%%%%%%%%%%%%%%%

\section{Introduction}
\indent This paper is devoted to the the following two-parameter nonlocal problem, namely $(P_{M,K,f}^{\mu,\lambda,h})$:
$$
\left\{\begin{array}{ll}
-M(\|u\|_{X_0}^2)\mathcal L_K u=\mu h\left(\displaystyle\int_\Omega \left(\int_0^{u(x)}f(x,t)dt\right)dx-\lambda\right)f(x,u) & \mbox{in}\,\,\Omega\\
u=0\,\,\,\mbox{in}\,\,\,\erre^n\setminus\Omega.
\end{array}\right.
$$

Here and in the sequel, $\Omega$ is a bounded domain in $({\R}^{n},|\cdot|)$ with $n>2s$ (where $s\in (0,1)$),
smooth (Lipschitz)
boundary $\partial \Omega$ and Lebesgue measure $|\Omega|$, $f:\Omega\times\erre\rightarrow \erre$ is a Carath\'{e}odory function with subcritical growth, $\lambda$ and $\mu$ are real parameters, $M,h$ are two suitable continuous functions and
$$
\|u\|_{X_0}^2:=\int_{\erre^n\times \erre^n}|u(x)-u(y)|^2K(x-y)dxdy.
$$

 \indent Further, $\mathcal L_K$ is a nonlocal operator defined as follows:
$$
\mathcal L_Ku(x):=
\int_{\erre^n}\Big(u(x+y)+u(x-y)-2u(x)\Big)K(y)dy,
\,\,\, (x\in \erre^n)
$$
where $K:\erre^n\setminus\{0\}\rightarrow(0,+\infty)$ is a function with the properties that:
\begin{itemize}
\item[$(\rm k_1)$] $\gamma K\in L^1(\erre^n)$, \textit{where} $\gamma(x):=\min \{|x|^2, 1\}$;
\item[$(\rm k_2)$] \textit{there exists} $\beta>0$
\textit{such that} $$K(x)\geq \beta |x|^{-(n+2s)},$$
{\textit{for any}} $x\in \erre^n \setminus\{0\}$;
\item[$(\rm k_3)$] $K(x)=K(-x)$, \textit{for any} $x\in \erre^n \setminus\{0\}$.
\end{itemize}

\indent A typical example of the kernel $K$ is given by $K(x):=|x|^{-(n+2s)}$. In this case $\mathcal L_K$ is the fractional Laplace operator defined as
$$
-(-\Delta)^s u(x):=
\int_{\erre^n}\frac{u(x+y)+u(x-y)-2u(x)}{|y|^{n+2s}}\,dy,
\,\,\,\,\, x\in \erre^n.
$$

\indent Problem $(P_{M,K,f}^{\mu,\lambda,h})$ is clearly highly nonlocal due to the presence of the fractional operator $\mathcal L_K$ and to the map $M$ as well as in the source term $f$. In our context, to avoid some additional technical difficulties originated by the presence of the term
$$M\Bigg(\displaystyle\int_{\erre^{n}\times\erre^n}|u(x)-u(y)|^2K(x-y)dxdy\Bigg),$$
we impose some restrictions on the behavior of $M$ (see Section 3).\par
 This setting includes the Kirchhoff-type problem of the form
$$
\left\{\begin{array}{ll}
-(a+b\|u\|_{X_0}^2)\mathcal L_K u=v(\mu,\lambda,h,f)&\mbox{in}\,\,\,\Omega\\
u=0 & \mbox{in}\,\,\,\erre^n\setminus\Omega,
\end{array}\right.
$$
where $a,b>0$ and
$$
v(\mu,\lambda,h,f):=\mu h\left(\displaystyle\int_\Omega \left(\int_0^{u(x)}f(x,t)dt\right)dx-\lambda\right)f(x,u),
$$
see Remark \ref{RK}.\par
\indent For completeness, in the vast literature on this subject, we refer the reader to some interesting recent results (in the non-fractional setting) obtained by Autuori and Pucci in \cite{AP1,AP2,AP3}
studying Kirchhoff equations by using different approaches.

We seek conditions on the data for which problem $(P_{M,K,f}^{\mu,\lambda,h})$ possesses at least three weak solutions. It is worth pointing out that the variational approach to attack such problems is not often easy to perform; indeed due to the presence of the nonlocal term, variational methods does not to work when applied to these classes of  equations.\par
 Fortunately, our approach here is realizable by checking that the associated energy functional (see Section 3) given by
$$
J_K(u):=\frac 1 2 \widehat{M}(\|u\|_{X_0}^2) -{\mu}H\left(\displaystyle\int_\Omega F(x,u(x))dx-\lambda\right),
$$

\noindent satisfies the assumptions requested by a recent critical point theorem (see Theorem \ref{CV} below) obtained by Ricceri in \cite[Theorem 1.6]{R3} and thanks to a suitable framework developed in \cite{svmountain}.\par

We emphasize that in \cite[Theorem 1.6]{R3} Ricceri established
a theorem tailor-made for a class of nonlocal problems involving
nonlinearities with bounded primitive. This result follows from \cite[Theorem 3]{R2} and the main novelty obtained in the most recent paper \cite{R3} is that, in contrast
with a large part of the existing literature, the abstract energy functional does not depend
on the parameter $\lambda$ in an affine way.\par

 The nonlocal analysis (see Section 2) that we perform here
in order to use Theorem \ref{CV} is quite general and has been successfully exploited for other goals
in several recent contributions; see \cite{sv,svmountain,svlinking,servadeivaldinociBN} and \cite{valpal}
for an elementary introduction to this topic and for a list of related references.\par
 In the nonlocal framework, the simplest example we can deal with is
given by the fractional Laplacian, according to the following
result.

\begin{theorem}\label{lapfra0}
Let $s\in (0,1)$, $n>2s$ and let $\Omega$ be an open bounded set of $\erre^n$
with
Lipschitz boundary. Moreover, let $f:\erre\rightarrow \erre$ be a non-zero continuous function such that
$$
\displaystyle\sup_{\xi\in\erre}|F(\xi)|<+\infty,
$$
where $
F(\xi):=\displaystyle\int_0^{\xi}f(t)\,dt,
$
for every $\xi\in \erre$. Further, let $$h: \left(-|\Omega|\displaystyle\osc_{\xi\in\erre}F(\xi),|\Omega|\displaystyle\osc_{\xi\in\erre}F(\xi)\right)\rightarrow \erre$$ be a continuous and non-decreasing function such that $h^{-1}(0)=\{0\}$.\par
 \indent Then, fixing $a,b>0$, for each $\mu$ sufficient large, there exists an open interval
 $$\Lambda\subseteq \left(|\Omega|\displaystyle\inf_{\xi\in\erre}F(\xi),|\Omega|\displaystyle\sup_{\xi\in\erre}F(\xi)\right)$$ and a number $\rho>0$ such that, for every $\lambda\in \Lambda$, the following equation
$$
\begin{array}{l} {\displaystyle \Big(a+b\displaystyle\int_{\erre^{2n}}\frac{|u(x)-u(y)|^2}{|x-y|^{n+2s}}\,dxdy\Big) \int_{\erre^{2n}}
\frac{(u(x)-u(y))(\varphi(x)-\varphi(y))}{|x-y|^{n+2s}}dxdy}\\
\displaystyle \qquad\qquad\qquad\qquad\qquad\quad=\mu h\left(\displaystyle\int_\Omega F(u(x))dx-\lambda\right)\int_\Omega f(u(x))\varphi(x)dx,\,\,\,\,\,\,\,
\end{array}
$$
for every
$$
\varphi\in H^s(\erre^n)\,\, {\rm such\,\, that}\,\, \varphi= 0\,\,\, {\rm a.e.\,\, in}\,\, \erre^n\setminus\Omega,
$$
has at least three distinct weak solutions $\{u_j\}_{j=1}^{3}\subset H^s(\erre^n)$,
such that $u_j=0$ a.e. in $\erre^n\setminus\Omega$, and
$$
\int_{\RR^n\times\RR^n}\frac{|u_j(x)-u_j(y)|^2}{|x-y|^{n+2s}}\,dxdy<\rho^2,
$$
for every $j\in \{1,2,3\}$.
\end{theorem}

\indent The plan of the paper is as follows. Section 2 is devoted to our abstract framework and preliminaries. Successively, in Section 3 we give the main result; see Theorem \ref{Esistenza}. Finally, a concrete example of an application is presented in Example \ref{esempio0}.\par
\indent We cite the monograph \cite{k2} for related topics on variational methods adopted in this paper and \cite{c1,c2,c3} for recent nice results in the fractional setting.

\section{Variational Framework}\label{section2}

In this subsection we briefly recall the definition of the functional space~$X_0$, firstly introduced in \cite{sv, svmountain}.
The reader familiar with this topic may skip this section and go directly to the next one. The functional space $X$ denotes the linear space of Lebesgue
measurable functions from $\RR^n$ to $\RR$ such that the restriction
to $\Omega$ of any function $g$ in $X$ belongs to $L^2(\Omega)$ and
$$
((x,y)\mapsto (g(x)-g(y))\sqrt{K(x-y)})\in
L^2\big((\RR^n\times\RR^n) \setminus ({\mathcal C}\Omega\times
{\mathcal C}\Omega), dxdy\big),$$ where ${\mathcal C}\Omega:=\RR^n
\setminus\Omega$. We denote by $X_0$ the following linear
subspace of $X$
$$X_0:=\big\{g\in X : g=0\,\, \mbox{a.e. in}\,\,
\RR^n\setminus
\Omega\big\}.$$
\indent We remark that $X$ and $X_0$ are non-empty, since $C^2_0 (\Omega)\subseteq X_0$ by \cite[Lemma~11]{sv}.\par
\indent Moreover, the space $X$ is endowed with the norm defined as
$$
\|g\|_X:=\|g\|_{L^2(\Omega)}+\Big(\int_Q |g(x)-g(y)|^2K(x-y)dxdy\Big)^{1/2}\,,
$$
where $Q:=(\RR^n\times\RR^n)\setminus \mathcal O$ and
${\mathcal{O}}:=({\mathcal{C}}\Omega)\times({\mathcal{C}}\Omega)\subset\RR^n\times\RR^n$. It is easily seen that $\|\cdot\|_X$ is a norm on $X$; see \cite{svmountain}.\par
 By \cite[Lemmas~6 and 7]{svmountain} in the sequel we can take the function
\begin{equation}\label{normaX0}
X_0\ni v\mapsto \|v\|_{X_0}:=\left(\int_Q|v(x)-v(y)|^2K(x-y)dxdy\right)^{1/2}
\end{equation}
as a norm on $X_0$. Also $\left(X_0, \|\cdot\|_{X_0}\right)$ is a Hilbert space with scalar product
$$
\langle u,v\rangle_{X_0}:=\int_Q
\big( u(x)-u(y)\big) \big( v(x)-v(y)\big)\,K(x-y)dxdy,
$$
see \cite[Lemma~7]{svmountain}.\par
 Note that in (\ref{normaX0}) (and in the related scalar product) the integral can be extended to all $\RR^n\times\RR^n$, since $v\in X_0$ (and so $v=0$ a.e. in $\RR^n\setminus \Omega$).\par

\indent While for a general kernel~$K$ satisfying
conditions from $(\rm k_1)$ to $(\rm k_3)$ we have that
$X_0\subset H^s(\RR^n)$, in the model case $K(x):=|x|^{-(n+2s)}$ the
space $X_0$ consists of all the functions of the usual fractional
Sobolev space $H^s(\RR^n)$ which vanish a.e. outside $\Omega$; see
\cite[Lemma~7]{servadeivaldinociBN}.\par
 Here $H^s(\RR^n)$ denotes
the usual fractional Sobolev space endowed with the norm (the
so-called \emph{Gagliardo norm})
$$
\|g\|_{H^s(\RR^n)}=\|g\|_{L^2(\RR^n)}+
\Big(\int_{\RR^n\times\RR^n}\frac{\,\,\,|g(x)-g(y)|^2}{|x-y|^{n+2s}}\,dxdy\Big)^{1/2}.
$$

\indent Before concluding this subsection, we recall the embedding
properties of~$X_0$ into the usual Lebesgue spaces; see
\cite[Lemma~8]{svmountain}. The embedding $j:X_0\hookrightarrow
L^{\nu}(\RR^n)$ is continuous for any $\nu\in [1,2^*]$, while it is
compact whenever $\nu\in [1,2^*)$, where $2^*:=2n/(n-2s)$ denotes the \textit{fractional critical Sobolev exponent}.\par

\indent For further details on the fractional Sobolev spaces we refer
to~\cite{valpal} and to the references therein, while for other
details on $X$ and $X_0$ we refer to \cite{sv}, where these
functional spaces were introduced, and also to \cite{sY, svmountain,
svlinking, servadeivaldinociBN}, where various properties of these spaces were
proved.\par
Finally, our abstract tool for proving the main result of the present paper is \cite[Theorem 1.6]{R3} that we recall here for reader's convenience.
 \begin{theorem}\label{CV}
 Let $(E,\|\cdot\|)$ be a separable and reflexive real Banach space and let $\eta, J:E\rightarrow \erre$ be two $C^1$-functionals with compact derivative and $J(0_E)=\eta(0_E)=0$. Assume also that $J$ is bounded and non-constant, and that $\eta$ is bounded above. Then, for every sequentially weakly lower semicontinuous and coercive $C^1$-functional $\psi:E\rightarrow \erre$ whose derivative admits a continuous inverse on $E^*$ and with $\psi(0_E)=0$, for every convex $C^1$-function $$\varphi:\left(-\displaystyle\osc_{u\in E}J(u),\displaystyle\osc_{u\in E}J(u)\right)\rightarrow [0,+\infty),$$ with $\varphi^{-1}(0)=\{0\}$, for which the number
 $$
 {\theta^{\star}}:=\displaystyle\inf_{u\in J^{-1}\left(\left(\displaystyle\inf_{u\in E}J(u), \displaystyle\sup_{u\in E}J(u)\right)\setminus \{0\}\right)}\frac{\psi(u)-\eta(u)}{\varphi(J(u))}
 $$
 is non-negative, and for every $\mu>\theta^{\star}$ there exists an open interval $$\Lambda\subseteq \left(\displaystyle\inf_{u\in E}J(u), \displaystyle\sup_{u\in E}J(u)\right)$$ and a number $\rho>0$ such that, for each $\lambda\in \Lambda$, the equation
 $$
 \psi'(u)=\mu\varphi'(J(u)-\lambda)J'(u)+\eta'(u)
 $$
 has at least three distinct solutions whose norms are less than $\rho$.
 \end{theorem}
 \begin{remark}\rm{
Note that, for a generic function $\psi:E\rightarrow\erre$, the symbol $\displaystyle\osc_{u\in E}\psi(u)$ denotes the number (possibly infinite) given by $$\displaystyle\osc_{u\in E}\psi(u):=\displaystyle\sup_{u\in E}\psi(u)-\displaystyle\inf_{u\in E}\psi(u).$$
Moreover, if $\psi$ is a $C^1$-functional, we say that the derivative $\psi'$ admits a continuous inverse on $E^*$ provided that there exists a continuous operator $T:E\rightarrow E^*$ such that
$$
T(\psi'(u))=u,
$$
for every $u\in E$.}
\end{remark}

\section{The Main Result}
Let $\mathcal{M}$ be the class of continuous functions $M:[0,+\infty)\rightarrow\erre$ such that:

\begin{itemize}
\item[$(C_M^1)$] $\displaystyle\inf_{t\geq 0} M(t)>0$;
\end{itemize}

\begin{itemize}
\item[$(C_M^2)$] \textit{there exists a continuous function $v_M: [0,+\infty)\rightarrow \erre$ such that
$$
v_M(tM(t^2))=t,
$$
for every $t\in [0,+\infty)$.}
\end{itemize}
\noindent Further, if $M\in \mathcal{M}$, set
$$\displaystyle
\widehat{M}(t):=\int_0^t M(s)ds,
$$
for every $t\in [0,+\infty)$.\par
Denote by $\mathcal{A}$ the class of all Carath\'{e}odory functions $f:\Omega\times\erre\rightarrow \erre$ such that
$$
\displaystyle\sup_{(x,t)\in\Omega\times\erre}\frac{|f(x,t)|}{1+|t|^{q-1}}<+\infty,
$$
for some $q\in [1,2^{*})$. Further, if $f\in \mathcal{A}$ we put
$$
F(x,\xi):=\displaystyle\int_0^{\xi}f(x,t)\,dt,
$$
for every $(x,\xi)\in \Omega\times\erre$.\par
 We recall that a \textit{weak solution} of problem $(P_{M,K,f}^{\mu,\lambda,h})$ is a function $u\in X_0$ such that
$$
\begin{array}{l} M(\|u\|_{X_0}^2){\displaystyle \int_Q \big(u(x)-u(y)\big)\big(\varphi(x)-\varphi(y)\big)K(x-y)dxdy}\\
\displaystyle \qquad\qquad\qquad\qquad=\mu h\left(\displaystyle\int_\Omega F(x,u(x))dx-\lambda\right)\int_\Omega f(x,u(x))\varphi(x)dx,\,\,\,\,\,\,\,
\end{array}
$$
for every $\varphi \in X_0$.\par
 For the proof of our result, we observe that
problem~$(P_{k,K,f}^{\mu,\lambda,h})$ has a variational structure, indeed it is the Euler-Lagrange equation of the functional $J_K:X_0\to \RR$ defined as follows
$$
J_K(u):=\frac 1 2 \widehat{M}(\|u\|_{X_0}^2) -{\mu}H\left(\displaystyle\int_\Omega F(x,u(x))dx-\lambda\right),
$$
where
$$\displaystyle
H(\xi):=\int_0^\xi h(t)dt,
$$
for every $\xi\in \erre$.\par

\indent Note that the functional $J_K$ is Fr\'echet differentiable in $u\in X_0$ and one has
$$
\langle J'_K(u), \varphi\rangle = M(\|u\|_{X_0}^2){\displaystyle \int_Q \big(u(x)-u(y)\big)\big(\varphi(x)-\varphi(y)\big)K(x-y)dxdy}
$$
$$
\qquad \qquad \qquad \qquad\, -h\left(\displaystyle\int_\Omega F(x,u(x))dx-\lambda\right)\int_\Omega f(x,u(x))\varphi(x)dx,
$$
for every $\varphi \in X_0$.\par
 Thus, critical points of $J_K$ are solutions to problem~$(P_{M,K,f}^{\mu,\lambda,h})$.
In order to find these critical points, we will make use of Theorem \ref{CV}.\par

Let us denote by
$$
\alpha_f:=\displaystyle\inf_{u\in X_0}\int_\Omega F(x,u(x))dx,\,\,\,\,\,\,\,\,\beta_f:=\displaystyle\sup_{u\in X_0}\int_\Omega F(x,u(x))dx,
$$
and
$$
\omega_f:=\beta_f-\alpha_f.
$$

\indent Finally, let
$$
\mathcal{R}:=\left\{u\in X_0, \int_\Omega F(x,u(x))dx \in (\alpha_f,\beta_f)\setminus\{0\}\right\}.
$$
\indent With the above notations our result reads as follows.

\begin{theorem}\label{Esistenza}
Let $s\in (0,1)$, $n>2s$ and let $\Omega$ be an open bounded set of
$\RR^n$ with Lipschitz boundary and
$K:\RR^n\setminus\{0\}\rightarrow(0,+\infty)$ be a map
satisfying $(\rm k_1)$--$(\rm k_3)$.
Moreover, let $f\in \mathcal{A}$ such that
\begin{equation}\label{F}
\sup_{(x,\xi)\in \Omega\times \erre}|F(x,\xi)|<+\infty,
\end{equation}
and
$$
\sup_{u\in X_0}\Big|\int_\Omega F(x,u(x))dx\Big|>0.
$$

\noindent Then, for every $M\in \mathcal{M}$ and every non-decreasing function $h:(-\omega_f,\omega_f)\rightarrow \erre$ with $h^{-1}(0)=\{0\}$, for every
$$
\mu>\inf_{u\in \mathcal{R}}\left\{\displaystyle\frac{\displaystyle  \widehat{M}(\|u\|_{X_0}^{2})}{\displaystyle 2H\left(\int_\Omega F(x,u(x))dx\right)}\right\}
$$
there exists an open interval $\Lambda\subseteq (\alpha_f,\beta_f)$ and a number $\rho>0$ such that, for each $\lambda\in \Lambda$, the problem~$(P_{k,K,f}^{\mu,\lambda,h})$
has at least three distinct weak solutions whose norms in $X_0$ are less than $\rho$.
\end{theorem}
\begin{proof}
Let us apply Theorem \ref{CV} by choosing $E:=X_0$, $\eta=0$, and
$$
 J(u):=\int_\Omega F(x,u(x))\,dx,\qquad \psi(u):=\frac 1 2 \widehat{M}(\|u\|_{X_0}^2)
$$
for every $u\in E$.\par
Since $f\in \mathcal{A}$, the functional $J$ is a $C^1$-functional with compact derivative (note that the embedding $j:E\hookrightarrow L^{q}(\Omega)$ is compact for every $q\in [1,2^*)$). Furthermore (by \eqref{F}) $J$ is clearly bounded.\par
 \indent Now, it is easy to see that $\psi$ is a $C^1$-funcitonal and, since $\widehat{M}$ is increasing, $\psi$ is also sequentially weakly lower semicontinuous.\par
  Let us prove that the derivative  $\psi':E\rightarrow E^*$ admits a continuous inverse. Since $E$ is reflexive, we identify $E$ with the topological dual $E^*$. For our goal, let $T:E\rightarrow E$ be the operator defined by
 $$
T(v):=\left\{\begin{array}{ll}
\displaystyle\frac{v_M(\|v\|_{X_0})}{\|v\|_{X_0}}v & \mbox{if}\, v\neq 0\\
0 & \mbox{if}\, v=0,
\end{array}\right.
$$
where $v_M$ appears in $(C_M^2)$.\par
 Thanks to the continuity of $v_M$ and $v_M(0)=0$, the operator $T$ is continuous in $E$.\par
 Moreover, for each $E\setminus\{0_E\}$, since $M(\|u\|^2_{X_0})>0$ (by $(C_M^1)$), one has
  $$
  T(\psi'(u))=T(M(\|u\|^2_{X_0})u)=\frac{v_M (M(\|u\|^2_{X_0})\|u\|_{X_0})}{M(\|u\|^2_{X_0})\|u\|_{X_0}}M(\|u\|^2_{X_0})u=u,
  $$
  \noindent as desired.\par
   Now, put
  $$
  \gamma:=\displaystyle\inf_{t\geq 0} M(t).
  $$
  So, $\gamma>0$ (by $(C_M^1)$) and
  $$
  \widehat{M}(t)\geq \gamma t,
  $$
  for every $t\in [0,+\infty)$. In particular, this implies that $\psi$ is coercive.\par
   In conclusion, let us take $\varphi:=H$. By our assumptions on $h$, it follows that the function $\varphi$ is non-negative, convex and $\varphi^{-1}(0)=\{0\}$.\par
    Then, the assertion of Theorem \ref{CV} follows and the existence of three weak solutions to our problem is established.
\end{proof}

\begin{remark}\label{Ricc}\rm{
Clearly, if the function $M$ is non-decreasing in $[0,+\infty)$, with $M(0)>0$, then the function $t\rightarrow tM(t^2)$ ($t\geq 0$) is increasing and onto $[0,+\infty)$, and so condition $(C_M^2)$ is satisfied. Taking into account the previous observations one can conclude that Theorem \ref{Esistenza} is the (non-perturbed) fractional analogous of \cite[Theorem 1.3]{R3} in which a nonlocal Dirichlet problem, in the classical framework,
was studied. If $g\in \mathcal{A}$, in analogy with the cited result, we point out that in Theorem \ref{Esistenza}, requiring that
$$
\sup_{(x,\xi)\in \Omega\times \erre}\max\left\{|F(x,\xi)|, \int_0^{\xi}g(x,t)dt\right\}<+\infty,
$$
\noindent instead of (\ref{F}), we have that for every $M\in \mathcal{M}$ and every non-decreasing function $h:(-\omega_f,\omega_f)\rightarrow \erre$ with $h^{-1}(0)=\{0\}$, for which the number
$$
\theta^*:=\displaystyle\inf_{u\in \mathcal{R}}\left\{\displaystyle\frac{\displaystyle  \widehat{M}(\|u\|_{X_0}^{2})-2\int_\Omega\left(\int_0^{u(x)}g(x,t)dt\right)dx}{\displaystyle 2H\left(\int_\Omega F(x,u(x))dx\right)}\right\}
$$
is non-negative, and for every $\mu>\theta^{*}$, there exists an open interval $\Lambda\subseteq (\alpha_f,\beta_f)$ and a number $\rho>0$ such that, for each $\lambda\in \Lambda$, the perturbed problem
$$
\left\{\begin{array}{ll}
-M(\|u\|_{X_0}^2)\mathcal L_K u=v(\mu,\lambda,h,f,g)&\mbox{in}\,\,\,\Omega\\
u=0 & \mbox{in}\,\,\,\erre^n\setminus\Omega,
\end{array}\right.
$$
where
$$
v(\mu,\lambda,h,f,g):=\mu h\left(\displaystyle\int_\Omega \left(\int_0^{u(x)}f(x,t)dt\right)dx-\lambda\right)f(x,u)+g(x,u)
$$
has at least three distinct weak solutions whose norms in $X_0$ are less than $\rho$.
}
\end{remark}

\begin{remark}\label{RK}\rm{
Fix $a,b>0$ and take
$$
M(t):=a+bt,
$$
for every $t\in [0,+\infty)$. Clearly condition $(C_M^1)$ and $(C_M^2)$ hold. Thus, as claimed in Introduction and bearing in mind Remark \ref{Ricc}, Theorem \ref{Esistenza} produces the existence of multiple weak solutions for the following fractional Kirchhoff-type problem depending on two
parameters:
$$
\left\{\begin{array}{ll}
-(a+b\|u\|_{X_0}^2)\mathcal L_K u=v(\mu,\lambda,h,f,g)&\mbox{in}\,\,\,\Omega\\
u=0 & \mbox{in}\,\,\,\erre^n\setminus\Omega.
\end{array}\right.
$$
}
\end{remark}

\begin{remark}\rm{Simple considerations explained in Section 2 prove that Theorem \ref{lapfra0} in Introduction is a consequence of Theorem \ref{Esistenza}.
}
\end{remark}

\indent In conclusion, we present a direct application of the main result.
\begin{es}\label{esempio0}\rm{
Let $s\in (0,1)$, $n>2s$ and let $\Omega$ be an open bounded set of $\erre^n$
with
Lipschitz boundary. Moreover, let $f:\erre\rightarrow \erre$ be the non-zero continuous function belonging to $\mathcal{A},$ with $$\displaystyle \sup_{\xi\in \erre}|F(\xi)|<+\infty$$ and let $M\in \mathcal{M}$. Then, owing to Theorem \ref{Esistenza}, for a sufficiently large $\mu$,
there exists an open interval $$\Lambda\subseteq \left(|\Omega|\displaystyle\inf_{\xi\in\erre}F(\xi),|\Omega|\displaystyle\sup_{\xi\in\erre}F(\xi)\right)$$ and a number $\rho>0$ such that, for each $\lambda\in \Lambda$, the following equation

$$
\begin{array}{l} M\left(\displaystyle\int_{\erre^{2n}}\frac{|u(x)-u(y)|^2}{|x-y|^{n+2s}}\,dxdy\right){\displaystyle \int_{\erre^{2n}}
\frac{(u(x)-u(y))(\varphi(x)-\varphi(y))}{|x-y|^{n+2s}}dxdy}\\
\displaystyle \qquad\qquad\qquad=\mu \frac{\left(\displaystyle\int_\Omega F(u(x))dx-\lambda\right)\displaystyle\int_\Omega f(u(x))\varphi(x)dx}{(|\Omega|\displaystyle\osc_{\xi\in\erre}F(\xi))^2-\left(\displaystyle\int_\Omega F(u(x))dx-\lambda\right)^2},\,\,\,\,\,\,\,
\end{array}
$$
for every
$$
\varphi\in H^{s}(\erre^n)\,\, {\rm such\,\, that}\,\, \varphi= 0\,\,\, {\rm a.e.\,\, in}\,\, \erre^n\setminus\Omega,
$$
has at least three distinct solutions $\{u_j\}_{j=1}^{3}\subset H^{s}(\erre^n)$,
such that $u_j=0$ a.e. in $\erre^n\setminus\Omega$, and
$$
\int_{\RR^n\times\RR^n}\frac{|u_j(x)-u_j(y)|^2}{|x-y|^{n+2s}}\,dxdy<\rho^2,
$$
\noindent for every $j\in \{1,2,3\}$.}
\end{es}
\begin{remark}\rm{We just observe that \cite[Theorem 3.1]{Molica} cannot be applied to the problem treated in the previous example.
}
\end{remark}

{\bf Acknowledgements.}  This paper was written when the first author was visiting professor at the University of Ljubljana in 2013. He expresses his gratitude to the host institution for warm hospitality.
  The manuscript was realized within the auspices of the GNAMPA Project 2013 titled {\it Problemi non-locali di tipo Laplaciano frazionario} and the SRA grants P1-0292-0101 and J1-5435-0101.

\end{document}